\documentclass{amsart}

\usepackage{amssymb}

\newtheorem{theorem}{Theorem}[section]

\newtheorem{lemma}[theorem]{Lemma}

\theoremstyle{definition}
\newtheorem{definition}[theorem]{Definition}

\theoremstyle{remark}

\numberwithin{equation}{section}

\def\R{\mathcal{R}}
\def\S{\mathfrak{S}}
\def\Z{\mathbb{Z}}
\def\N{\mathbb{N}}
\def\zmin{\mu}
\def\zmax{\nu}
\def\dmax{\kappa}
\def\hexpl#1{\hbox{ $\langle$\textit{#1}$\rangle$}}

\begin{document}

\title{Two Useful Facts About Generating Functions}


\author{Alex Kasman}
\address{Department of Mathematics / College of Charleston /
  Charleston SC 29401}
\email{kasmana@cofc.edu}

\author{Robert Milson}
\address{Department of Mathematics and Statistics, Dalhousie
  University, Halifax, NS B3H~3J5}
\email{rmilson@dal.ca}


\date{\today}

\begin{abstract}
Sequences are often conveniently encoded in the form of a generating
function depending on a formal variable.  
This note presents two observations that allow one to draw conclusions
about the generated sequence from the generating function.  The first
constructively produces ``recursion relations'' for the sequence from
differential operators in the formal variable having the generating function as an
eigenfunction.  The second allows one to determine whether the
sequence is orthogonal with respect to some inner product by
considering the result of taking the inner product of the generating
function with itself. Examples
presented to demonstrate the use and value of these methods
include a sequence of numbers, 
a family of Exceptional Hermite Polynomials, and an example
illustrating the result in a non-commutative setting.

\medskip

\raggedright
\noindent\textit{This paper is the record of a talk given by the first author at the AMS  Special Session on Orthogonal Polynomials and their Applications at the JMM held in Boston, MA in January 2023.  A version of it was first published by the American Mathematical Society as Contemporary Mathematics 822 (2025) pp.~155-166.}

\end{abstract}

\maketitle

\section{Introduction}

It is often convenient to encode an infinite sequence
 $\{a_0,a_1,a_2,\ldots\}$ in the form of a generating function
 $\psi(z)$ which has the sequence as the coefficients in a power
 series representation.
 In particular, when $\psi(z)$ is an ordinary generating function
 which can be written in a closed form that is
 analytic at $z=0$ it is then possible to recover elements of the
 sequence as
 $$
 a_n=\frac{1}{n!}\frac{d^n}{dz^n}\psi(z)\bigg|_{z=0}.
 $$
 
Generating functions are common in many areas of pure and
 applied mathematics, with the variable $z$ usually seen as nothing more
 than a formal variable with no function aside from encoding the sequence.  As George P\'olya wrote

 \medskip
 
 \begin{quote}
\textit{A generating function is a device somewhat similar to a bag. Instead of carrying many little objects detachedly, which could be embarrassing, we put them all in a bag, and then we have only one object to carry, the bag}
\cite{Polya}.
\end{quote}

 \medskip

\noindent
However, the formal variable in the generating function can play
a more useful and interesting role.

The main result of this note (Theorem~\ref{thm:main}) is a correspondence between differential
operators in $z$ having the generating function $\psi(z)$ as an eigenfunction
and difference operators acting with the same eigenvalue on the
generated sequence.  A second result (Theorem~\ref{thm:orthog})
allows one to determine whether the sequence is orthogonal by applying
that inner product to the generating function itself\footnote{After
  this paper was submitted, it was pointed out to the authors (by
  another participant in this workshop) that the results of this
  theorem were implicitly used in \cite{usedit}.  Consequently, we
  cannot claim this result as original.  However, it appears not to be
  very well-known and we hope its inclusion here will help spread the
  word about this interesting way to demonstrate orthogonality.}.

Examples demonstrating both of these techniques are presented in
Section~\ref{sec:examples}.  It is our hope
that others are able to apply these results to the generating
functions that arise in their own research.

\subsection{Notation and Definitions}

Let $\R$ be a ring and $\R_0\subseteq\R$ be any subring. We 
will usually be interested in
the case that the elements of $\R$ are functions of a variable $x$ and
$\R_0$ consists of just the constant functions.
No assumptions are made about commutativity in $\R$ aside from
the fact that the formal variable $z$ introduced below commutes with all elements of
the ring.

The $\R$-valued function $\psi(z)$ 
of the complex variable $z$  is assumed to be analytic at $z=0$ with
power series representation
\begin{equation}
\psi(z)=\sum_{n=0}^{\infty} a_n z^n.\label{eqn:regulargf}
\end{equation}
(This necessarily implies the existence of a suitable topology on
$\R$, but it will play no role here besides allowing for the assumed
convergence of the infinite sum in $(\ref{eqn:regulargf})$, and so no
more will be said about it.)

Associated to the choice of the generating function $\psi(z)$ we have
infinite sequence $\{a_0,a_1,a_2,\ldots\}$ with $a_n\in\R$ for each $n\in\N\cup\{0\}$
define by the formula
 $\psi^{(n)}(0)=n! a_n$.
For later
convenience, we define $a_n=0$ to be the zero element of $\R$
when $n<0$.

The monomial differential operator $z^i \partial_z^j$ for
$i\in\Z$ and $j\in\N\cup\{0\}$ acts on a sufficiently differentiable $\R$-valued function 
$\gamma(z)$ 
according to the formula
\begin{equation}
z^i\partial_z^j(\gamma(z))=z^i\frac{\partial^j}{\partial z^j}(\gamma(z))=z^i\gamma^{(j)}(z).
\label{eqn:monomiald}
\end{equation}
The ring $\R_0[z,z^{-1},\partial_z]$ is made up of all finite linear
combinations of such monomial differential operators with ``constant''
coefficients
from $\R_0$ and their action can be extended linearly from
the action of those basis elements in \eqref{eqn:monomiald}.

If there exist 
$\lambda\in\R$ and $L\in \R_0[z,z^{-1},\partial_z]$ such that
\begin{eqnarray}
\label{eigenL}L(\psi(z))&=&\lambda\psi(z)\\\hbox{or}\qquad
L(\psi(z))&=&\psi(z)\lambda.\label{eigenR}
\end{eqnarray}
then we say that $\psi(z)$ is an eigenfunction for the differential
operator $L$ with eigenvalue $\lambda$.  More specifically, in the
non-commutative case we say that $\lambda$ is a left eigenvalue when
\eqref{eigenL} holds and that it is a right eigenvalue when
\eqref{eigenR} holds.

Consider any function
$f:\N\cup\{0\}\to\R_0$ with values in the ``ring of
constants'' $\R_0$ and any integer $i\in \Z$.
The corresponding elementary difference operator $f \S^i$ acts on
 the sequence
$\{a_n\}$ turning it into a new sequence of elements from $\R$ according to the rule
\begin{equation}
  f\S^i(a)_n=f(n)a_{n+i}.
\end{equation}
(which is zero if $n+i<0$ by an earlier assumption).  Any finite sum
of these elementary difference operators is a difference operator
whose action is extended linearly from the formula above.
If $\Omega$
is a difference operator satisfying $\Omega(a)_n=\lambda a_n$ (or alternatively
$\Omega(a)_n=a_n\lambda$) for every $n\in \N\cup\{0\}$ and some
fixed $\lambda\in\R$ then we say that $\{a_n\}$ is an eigensequence
for $\Omega$ with left (or alternatively right) eigenvalue $\lambda$.

\section{Difference Equations from Generating Eigenfunctions}
This section concerns a certain
difference operator $\Omega_L$ corresponding to the choice of any
differential operator $L\in\R_0[z,z^{-1},\partial_z]$:
\begin{definition}\label{def:LOmega}
Suppose $L\in \R_0[z,z^{-1},\partial_z]$ is an ordinary differential
operator in $z$ with singularities only at $z=0$.   Then it can be written in the form
$$
L=\sum_{i=\zmin}^{\zmax}\sum_{j=0}^{\dmax} c_{ij}z^{i}\partial_z^{j},
\qquad
\zmin<\zmax\in\Z,\ \dmax\in\N\cup\{0\}, c_{ij}\in\R_0.
$$
Associate to the choice of any such differential operator a
corresponding \textit{difference} operator $\Omega_L$ defined as
\begin{eqnarray*}
\Omega_L&=&\sum_{i=\zmin}^{\zmax}\sum_{j=0}^{\dmax} c_{ij}
\prod_{k=0}^{j-1}(n+j-i-k) \S^{j-i}  \\
  &=&\sum_{i=\zmin}^{\zmax}\sum_{j=0}^{\dmax} c_{ij}
(n+j-i)(n+j-i-1)\cdots(n-i+1) \S^{j-i}.
\end{eqnarray*}
\end{definition}

  \subsection{General Lemma}
  Before the main result which applies in the special case
  that $\psi(z)$ is an eigenfunction of $L$, we first state
a more general result showing a correspondence between the action of
those two operators \textit{without} making any assumptions about the
generating function.
As this lemma shows, 
the action of $L$ on any generating function $\psi$ can
    be elegantly described in terms of the corresponding difference operator $\Omega_L$:

\begin{lemma}\label{lem:main} For any differential operator $L\in\R_0[z,z^{-1},\partial_z]$ and
  the corresponding difference operator $\Omega_L$ as in
  Definition~\ref{def:LOmega} we have
$$
L(\psi)=\sum_{n=\zmin}^{\infty} \Omega_L(a)_nz^n.
$$
\end{lemma}
\begin{proof}
For $j,m\in\N\cup\{0\}$ one has
\begin{equation}\label{eqn:simple}
\partial_z^{j}(z^m)
=
m(m-1)\cdots(m-j+1)z^{m-j}.
\end{equation}
Note that this formula is valid even in the case that $m<j$ because in
that case $z^m$ is in the kernel of the operator and a factor of
$(m-m)$ appears in the product on the right side, so they are both zero.

It then follows by linearity that
\begin{eqnarray*}
  L(\psi)&=&\sum_{i=\zmin}^{\zmax}\sum_{j=0}^{\dmax}c_{ij}z^i\partial_z^j\psi(z)\\
         &=&\sum_{i=\zmin}^{\zmax}\sum_{j=0}^{\dmax}
             c_{ij}z^i\partial_z^j\left( \sum_{m=0}^\infty a_m z^m\right)
  \hexpl{expanding $\psi$ using the index $m$}\\
         &=&\sum_{i=\zmin}^{\zmax}\sum_{j=0}^{\dmax}\sum_{m=0}^\infty c_{ij}a_mz^i\left(\partial_z^jz^m\right)\\
  &=&\sum_{i=\zmin}^{\zmax}\sum_{j=0}^{\dmax}\sum_{m=0}^\infty c_{ij}a_mz^i\left( m(m-1)\cdots(m-j+1)z^{m-j}\right)
\hexpl{by \eqref{eqn:simple}}\\
\phantom{L(\psi)}  &=&\sum_{i=\zmin}^{\zmax}\sum_{j=0}^{\dmax}\sum_{m=0}^\infty c_{ij}a_m
      m(m-1)\cdots(m-j+1)z^{m-j+i}\\
  &=&\sum_{i=\zmin}^{\zmax}\sum_{j=0}^{\dmax}\sum_{n=i-j}^\infty c_{ij}a_{n+j-i}
      (n+j-i)(n+j-i-1)\cdots(n-i+1)z^{n}.
\end{eqnarray*}

The final equality above is obtained by rewriting the inner-most sum
in terms of the new variable $n=m-j+i$.
After doing so, this sum begins at $n=i-j$.  However, the equality would
still be true if it started at a value \textit{lower} than $i-j$ since
$a_n=0$ for $n<0$.  In addition, the smallest power of $z$ which can
appear in this expression with a non-zero coefficient is $z^{\zmin}$
because $\psi$ is expanded only in non-negative powers of $z$, their
derivatives only involve non-negative powers of $z$, and those are
multiplied by $z^i$ for $i\geq \zmin$.  So, the sum in $n$ can be
taken to start at $n=\zmin$.  Then, since there are only a finite
number of terms in these sums having $z^n$ for any fixed value of $n$,
there can be no issues about reordering that would prevent us from
reordering the sums.
Consequently, we have
\begin{eqnarray*}
L(\psi)& =&\sum_{i=\zmin}^{\zmax}\sum_{j=0}^{\dmax}\sum_{n=\zmin}^\infty c_{ij}a_{n+j-i}
      (n+j-i)(n+j-i-1)\cdots(n-i+1)z^{n}\\
&=&\sum_{n=\zmin}^{\infty}\left(
\sum_{i=\zmin}^{\zmax}\sum_{j=0}^{\dmax} 
c_{ij}(n+j-i)(n+j-i-1)\cdots(n-i+1)a_{n+j-i}\right)z^{n}.\\
\end{eqnarray*}
The claim follows from the observation that the coefficient of each
term $z^n$
in the formula above is precisely $\Omega_L(a)_n$.
\end{proof}

For example, note that $L=z\partial_z$
corresponds to the difference operator $\Omega_L=n$ which simply
multiplies $a_n$ by its index.  That makes sense because
$$
z\partial_z(\psi(z))=z\left(\sum_{n=0}^{\infty} na_nz^{n-1}\right)
=\sum_{n=0}^{\infty}na_n z^{n}.
$$
This is true regardless of the series $\{a_n\}$.
However, our primary interest below will be on the case in which
$\psi(z)$ is an \textit{eigenfunction} for a differential operator
$L$, and whether that is true clearly depends not only on the choice of operator but
also on the sequence.

\subsection{When the Generating Functions is an Eigenfunction of $L$}

Again consider a differential operator $L\in\R_0[z,z^{-1},\partial_z]$
and its corresponding difference operator $\Omega_L$ as in
Definition~\ref{def:LOmega}.  However, we now further assume that the
generating function $$\psi(z)=\sum_{n=0}^{\infty}a_nz^n$$
is an eigenfunction for $L$ satisfying either $L(\psi)=\lambda\psi$ or
$L(\psi)=\psi\lambda$ for some $\lambda\in\R$.
As an easy corollary of Lemma~\ref{lem:main} we then get:

\begin{theorem}\label{thm:main}
If the generating function $\psi(z)$ is an eigenfunction for $L$ with
(left or right) eigenvalue $\lambda$ then the generated sequence $\{a_0,a_1,a_2,\ldots\}$ is an eigensequence
for the difference operator $\Omega_L$ with (left or right) eigenvalue $\lambda$.
\end{theorem}
\begin{proof}
Suppose $L(\psi)=\lambda\psi$.
Then
  $$
  L(\psi)=\lambda\psi=\sum_{n=0}^{\infty}\lambda a_nz^n
  $$
  and from the lemma we have
$$
L(\psi)=\sum_{n=\zmin}^{\infty}\Omega_L(a)_nz^n.
$$
Setting the coefficients equal we get that $\Omega_L(a)_n=0$ for $n<0$
and $\Omega_L(a)_n=\lambda a_n$ for $n\geq0$, which is precisely the
definition of being an eigensequence.  The same argument also
works with the eigenvalue $\lambda$ appearing on the right.
\end{proof}

\section{Determining Orthogonality Using Generating Functions}
Suppose that $\R$ is actually a \relax{vector space} with inner
product $\langle\cdot,\cdot\rangle_\R$.  A natural question to then ask
about the sequence $\{a_0,a_1,a_2,\ldots\}$ is whether it is
\textit{orthogonal} in the sense that there exists some sequence
$\{\omega_0,\omega_1,\ldots\}$ so that
\begin{equation}\label{eqn:orthogdef}
\langle a_m,a_n\rangle_\R=\omega_n\delta_{mn}
\end{equation}
for all $m,n\in\N\cup\{0\}$.
The following result provides a means to answer that question using
the generating function:
\begin{theorem}\label{thm:orthog}
Let $\psi(z)$ be a regular generating function for the sequence
$\{a_0,a_1,\ldots\}$ as in \eqref{eqn:regulargf}.  Then the sequence
is orthogonal with respect to the inner-product
$\langle\cdot,\cdot\rangle_\R$ if and only if there exists a function
$\theta(z)$ such that
$$
\langle \psi(w),\psi(z)\rangle_\R=\theta(wz).
$$
Furthermore, if that is the case, then the function $\theta(z)$ is a
regular generating function for the sequence $\{\omega_0,\omega_1,\ldots\}$.
\end{theorem}

\begin{proof}
Expanding the generating functions and distributing\footnote{To avoid
  worrying about complex conjugates, we
  assume $w,z\in\mathbb{R}$.} the inner product
we have
\begin{eqnarray*}
  \langle \psi(w),\psi(z)\rangle_\R&=&
                                               \left\langle \sum_{m=0}^\infty
                                               a_mw^m,\sum_{n=1}^{\infty}
                                               a_nz^n\right\rangle_\R\\
  &=&\sum_{m=0}^\infty\sum_{n=0}^\infty\langle a_m,a_n\rangle_\R w^mz^n.
\end{eqnarray*}
If the sequence satisfies \eqref{eqn:orthogdef} then the terms for which $m\not=n$ have
a coefficient of zero and the result is
\begin{equation}
  \langle \psi(w),\psi(z)\rangle_\R=\sum_{n=0}^\infty
  \omega_{n}(wz)^n=\theta(wz)\label{eqn:inprod}
  ,
\end{equation}
  which as claimed is a function of the product $wz$ that is a
  generating function for the sequence $\{\omega_0,\omega_1,\ldots\}$.

On the other hand, if the sequence $\{a_0,a_1,\ldots\}$ is
\textit{not} orthogonal, then the right hand side of the expression
\eqref{eqn:inprod} involves a product of the form $w^mz^n$ where
$m\not=n$ and thus is not of the form $\theta(wz)$ for any function $\theta$.
\end{proof}

Section~\ref{sec:examp2} below includes an example in which this result
is used to demonstrate the orthogonality of a sequence of polynomials.

\section{Examples}\label{sec:examples}
\subsection{A Sequence of Numbers}
The function $\psi(z)=(z-1)e^z$ is an eigenfunction for the
differential operator $L=\partial_z^2-\frac{2}{z}\partial_z$ with
eigenvalue $\lambda=1$.
In this simplest example where $\R=\R_0=\mathbb{R}$ we will just get
from this eigenfunction a
difference equation satisfied by a sequence of real numbers.

Expanding the function as a power series in $z$ we see that
$$
\psi(z)=-1+\frac{z^2}{2}+\frac{z^3}{3}+\frac{z^4}{8}+\frac{z^5}{30}+\frac{z^6}{144}+\cdots
$$
and hence this is a generating function for a sequence which begins
$a_0=-1$, $a_1=0$, 
$a_2=\frac12$, $a_3=\frac13$, $a_4=\frac18$, $a_5=\frac1{30}$.

``Translating'' the differential operator $L$ into a difference
operator according to Definition~\ref{def:LOmega} we get
$$
\Omega_L=(n^2+n-2)\S^2.
$$
Theorem~\ref{thm:main} then assures us that applying this difference operator to the elements of the sequence will multiply
them by $\lambda=1$ and hence leave them unchanged.  Indeed, we can
verify that this is the case for the terms at the start of the sequence
since:
$$
\Omega_L(a)_0=(0^2+0-2)a_2=-2\times\frac12=-1=a_0,
$$
$$\Omega_L(a)_1=(1^2+1-2)a_3=0\times\frac13=0=a_1,
$$
$$
\Omega_L(a)_2=(2^2+2-2)a_4=4\times\frac18=\frac12=a_2,
$$
and so on.

\subsection{Exceptional Hermite Polynomials}\label{sec:examp2}
As readers of this volume of Contemporary Mathematics surely know, the classical
orthogonal polynomials are important examples of sequences that are
often written in terms of generating functions.  Moreover, the
three-term recursion relations that follow as a consequence of their
orthogonality make them eigensequences for difference operators, in
the terminology of this paper.

Exceptional orthogonal polynomials are a generalization of the
classical orthogonal polynomials in which a finite number of positive
integer degrees are missing \cite{UllateKamranMilson2009}.  The
Exceptional Hermite Polynomials (XHPs) are a special class which was first
studied in 2014 \cite{UllateGrandatiMilsonXHPS}.  In fact, there is a
family of XHPs associated to each Young diagram. More recently, it
was found that for every Young diagram there is a  KP Wave Function
from $Gr_0$ that is a generating function for the corresponding family
of XHPs \cite{MilsonKasman2020}.  Moreover, the wave functions from $Gr_0$ are all
eigenfunctions for differential operators in $\mathbb{C}[z,z^{-1},\partial_z]$
\cite{Wilson1993}.  Here we will consider just one example to
illustrate the two main theorems of this note.


The ordinary generating function
$$
\psi(x,z)= \left(2 x^2 z^2-4 x z+z^2+4\right) e^{x z-\frac{z^2}{4}}
$$
produces the sequence:
$$
a_0=4,\ a_1=a_2=0,\ a_3=\frac{2 x^3}{3}+x,\
a_4=\frac{x^4}{2}+\frac{x^2}{2}-\frac{1}{8},\
a_5=\frac{x^5}{5}-\frac{x}{4},\cdots
$$
which includes a polynomial of every degree \textit{except} degrees
$1$ and $2$.


Now, we will use Theorem~\ref{thm:orthog} to demonstrate the
orthogonality of this sequence of polynomials with respect to the
inner product
   \[
   \langle p(x),q(x)\rangle
   =\int_{-\infty}^{\infty}p(x)q(x)\frac{e^{-x^2}}{(x^2+1/2)^2}\,dx.
 \]
  To that end, set
  \begin{gather*}
    S=w+z,\quad P= \frac12 wz-1\\
    \phi(x,z) = \frac{z^2}{2}+ \frac{1-xz}{x^2+1/2}.
  \end{gather*}
An elementary calculation shows that
\begin{equation}
  \partial_x \left[ \frac{2xP-S}{2x^2+1} e^{x(S-x)}\right] =
-  \left(\phi(x,w)\phi(x,z) - P^2-1\right) e^{x(S-x)}.
\end{equation}
It follows that,
\[ \int_{-\infty}^\infty
  \phi(x,w)\phi(x,z) e^{x(S-x)}\,dx = ( P^2+1)\int_{-\infty}^\infty
  e^{x(S-x)}dx =( P^2+1) e^{\frac{S^2}4}\sqrt\pi.\]
Since
\[ \frac{\psi(x,z)}{x^2+1/2} = 4 \phi(x,z) e^{xz - z^2/4} \]
we have that
\begin{gather*}
 \psi(x,w) \psi(x,z) \frac{e^{-x^2}}{(x^2+1/2)^2}  = 16
 \phi(x,w)\phi(x,z) e^{x(S-x)} e^{-\frac{w^2}4-\frac{z^2}4}\\
\hbox{and }   \int_{-\infty}^\infty  \psi(x,w) \psi(x,z)
   \frac{e^{-x^2}}{(x^2+1/2)^2} \,dx = 4e^{\frac{wz}{2} }(w^2z^2-4 w z+8) \sqrt{\pi}.
 \end{gather*}
That expression depends only on the product $wz$.  Hence, it follows from
Theorem~\ref{thm:orthog} that
$$
\langle a_m,a_n\rangle=\left.\frac{d^n}{dz^n}\frac{4e^{\frac{z}{2} }(z^2-4  z+8) \sqrt{\pi}}{n!}\right|_{z=0}\,\delta_{mn}= \frac{16(n-1)(n-2)}{2^n n!}\sqrt\pi\,\delta_{mn}.
$$

Because of the orthogonality, it is not surprising that the
polynomials satisfy a set of ``recursion relations'', and because of
the missing degrees it is not surprising that it involves a
commutative algebra of higher-order difference operators.  But, it is
not immediately clear what is the corresponding algebra of difference
operators.

To elucidate the commutative algebra of difference operators and the
corresponding eigenvalues, let us observe that
\[ \partial_x \psi(x,z) = (1+2x^2)z^3 e^{xz-\frac{z^2}4}.\] Let $\R$
denote the codimension-2 vector space of polynomials $\lambda(x)$ with
the property that $\lambda'(x)$ is divisible by $1+2x^2$.  In fact,
$\R$ is a ring because if $f,g\in\R$, then $(fg)' = f' g + f g'$ is
also divisible by $1+2x^2$.  As a vector space, $\R$ is
infinite-dimensional.  However, as a commutative ring, $\R$ is
generated by
\[ \lambda_0 = 1, \quad \lambda_3 = x^3+\frac32 x,\quad
  \lambda_4 = x^4+x^2,\quad \lambda_5 =  x^5 -\frac54 x.\]

Thus, $\psi(x,z)$ may be regarded as a kind of generating function for
$\R$ with the coefficients being a basis that is orthogonal with
respect to the above inner-product.  Since $\R$ is a ring, for every
$n\ne 1,2$ the coefficients of $\lambda_n(x) \psi(x,z)$ satisfies the
same defining property, namely their derivative is divisible by
$1+2x^2$.   It is therefore very plausible that the coefficients of
$\lambda_n(x) \psi(x,z)$ can be obtained by applying a difference
operator to the orthogonal sequence $a_n(x)$.   However, the key to
obtaining these difference operators is
Theorem~\ref{thm:main} together with the fact
that for every $n\ne 1,2$ we have that
\[ L_n(z,\partial_z)\psi(x,z) = \lambda_n(x) \psi(x,z) \] where
$L_n(z,\partial_z)$ is a differential operator of order $n$.  This
immediately implies that the $L_n,\; n\ne 1,2$ span a commutative
algebra of differential operators.

The procedure for constructing the operators $L_n$ and the
corresponding difference operators is given in
\cite{MilsonKasman2020}.  Here are the operators corresponding to
eigenvalues $\lambda_3,\lambda_4, \lambda_5$, as defined above:
\begin{align*}
  L_3
  &= \partial_z^3 + \left( \frac32 z-6 z^{-1}\right) \partial_z^2 +
    \left( \frac34 z^2 -3 + 12 z^{-2}\right) \partial_z + \frac18
    z^3,\\
  L_4
  &= \partial_z^4+\left(2z-8z^{-1}\right) \partial_z^3 + \left(
    \frac32 z^2 -8 + 28 z^2 \right) \partial_z^2, \\
  &\qquad + \left(\frac12 z^3
    -2z + 12 z^{-1} -40 z^{-3}\right) \partial_z +
    \frac1{16}z^4+\frac14\\
  L_5
  &= 
    \partial_z^5+
    \left(\frac52z-10z^{-1}\right) \partial_z^4+
    \left(\frac52z^2-15+50z^{-2}\right) \partial_z^3\\
  &\quad+
    \left(\frac54 z^3-\frac{15}{2} z+
    45z^{-1}-140z^{-3}\right) \partial_z^2\\
  &\quad+ \left(\frac{5}{16} z^4
    -\frac54 z^2+10-60z^{-2}+180z^{-4}\right)
    \partial_z +\frac{1}{32}z^5.
\end{align*}
Each of the above differential operators $L_k$ of order $k$
corresponds to a difference operator $\Omega_{k} = \Omega_{L_k}$ of
order $2k$.  These, difference operators commute and generate a ring
of  recurrence relations for the exceptional Hermite polynomials:
\begin{equation}
  \label{eq:recrel}
\Omega_{k}(a)_n=\lambda_k a_n.
\end{equation}
Using the construction described  in
Theorem~\ref{thm:main} the difference operators in question are:
\begin{align*}
  \Omega_3
  &=(n-2) (n-1) (n+3)
    \S^3+\frac{3}{2} (n-2) (n+1) \S^1+\frac{3}{4} (n-1)
    \S^{-1}+\frac18\S^{-3},\\
  \Omega_4
  &=
    (n-2) (n-1) (n+1) (n+4)
    \S^4+  2 (n-2) (n-1) 
    (n+2) \S^2\\
  &\quad +  \frac{1}{4} \left(6 n^2-14 n+1\right) 
    +\frac{1}{2}
    (n-2) \S^{-2}+\frac{1}{16}\S^{-4},\\
  \Omega_5
  &=
    (n-1) (n+1) (n+2) 
    (n+5) (n-2) \S^5+\frac{5}{2} (n-1) n (n+3) (n-2)
    \S^3\\
  &\quad +
    \frac{5}{2} (n+1) (n-2)^2 \S^1 +\frac{5}{16} (n-3)
    \S^{-3}+\frac{5}{4} (n-3) (n-1) 
    \S^{-1}+\frac{1}{32}\S^{-5}.
\end{align*}

The Classical Hermite polynomials $h_n(x),\; n=0,1,2,\ldots$ can also be
defined as the coefficients of the classical generating function
\[\psi_0(x,z) = e^{xz-\frac14 z^2} = \sum_{n=0}^\infty h_n(x)
  \frac{z^n}{n!},\]
and by means of the classical recurrence relation
\[ x h_n(x) = \frac12 h_{n+1}(x) + n h_{n-1}.\] Analogously, the
exceptional Hermite polynomials $a_n(x)$ can be recursively defined by
means of the recurrence relations \eqref{eq:recrel}.  Indeed,
the recurrences
\begin{align*}
  \lambda_3 a_0 &= \Omega_3(a)_0 = 6 a_3\\
  \lambda_4 a_0 & = \Omega_4(a)_0 = 8 a_4+ \frac14 a_0 \\
  \lambda_5 a_0 &= \Omega_5(a)_0 = 20 a_5\\
  \lambda_3 a_3 &= \Omega_3(a)_3 = 12 a_6+6 a_4 + \frac18 a_0
\end{align*}
serve to define $a_3,a_4,a_5,a_6$.  However, $a_7$ can be defined
recursively in two ways:
\begin{align*}
  \lambda_4 a_3 &= \Omega_4(a)_3 = 56 a_7+20 a_5 + \frac{13}4 a_3\\
  \lambda_3 a_4 &= \Omega_3(a)_4 = 42 a_7+15 a_5+\frac94 a_3.
\end{align*}
The fact that these two definitions give the same value of $a_7$ is a
consequence of the fact that the $\Omega_3,\Omega_4$ commute,
which in turn is a consequence of the fact that $L_3$ and $L_4$ commute.


Although this procedure can be applied to any of the families of XHPs, we will
do so here in just one case as an example.  (For information about the
general case, consult \cite{MilsonKasman2020} or the paper by Luke
Paluso in this same volume of Contemporary Mathematics.)

\subsection{Non-Commutative Example}\label{sec:examp3}


The function
$$
\psi=\left(
\begin{array}{cc}
 x z-1 & -z^2 \\
 0 & x z \\
\end{array}
\right)e^{xz}
$$
is a generating function for a sequence of matrix polynomials which
begins
$$
a_0= \left(
\begin{array}{cc}
 -1 & 0 \\
 0 & 0\\
\end{array}
\right),
a_1= \left(
\begin{array}{cc}
 0 & 0 \\
 0 & x \\
\end{array}
\right),
a_2=\left(
\begin{array}{cc}
 \frac{x^2}{2} & -1 \\
 0 & x^2 \\
\end{array}
\right)
,\ a_3=\left(
\begin{array}{cc}
 \frac{x^3}{3} & -x \\
 0 & \frac{x^3}{2} \\
\end{array}
\right).
$$
It happens to be an eigenfunction for
the matrix differential operator
\begin{eqnarray*}
L&=&\left(
\begin{array}{cc}
 1 & 0 \\
 0 & 1 \\
\end{array}
\right)\partial_z^3
+
\left(
\begin{array}{cc}
 -\frac{3}{z} & 0 \\
 0 & -\frac{3}{z}
\end{array}
\right)\partial_z^2
+
\left(
\begin{array}{cc}
 \frac{3}{z^2} & 3 \\
 0 & \frac{6}{z^2}\\
\end{array}
\right)\partial_z
+
\left(
\begin{array}{cc}
 0 & -\frac{6}{z} \\
 0 & -\frac{6}{z^3}\\
\end{array}
\right) 
\end{eqnarray*}
satisfying the eigenvalue equation
$
L\psi=x^3\psi.
$

Theorem~\ref{thm:main} then assures us that the generated sequence is
an eigensequence for the difference operator
\begin{eqnarray*}
\Omega_L&=&\left(
\begin{array}{cc}
 n^3+3 n^2-n-3 & 0 \\
 0 & n \left(n^2+3 n+2\right) \\
\end{array}
\right)
\S^3+\left(
\begin{array}{cc}
 0 & 3 (n-1) \\
 0 & 0 \\
\end{array}
\right)
\S
\end{eqnarray*}
with the same eigenvalue $x^3$.
For example, applying the difference operator to $a_2$ we get
$$
\Omega_L(a)_2=\left(
\begin{array}{cc}
 15 &0 \\
 0 & 24 \\
\end{array}
\right)a_5+\left(
\begin{array}{cc}
0 &3 \\
 0 & 0 \\
\end{array}
\right)a_3
=\left(
\begin{array}{cc}
 \frac{x^5}{2} & -x^3 \\
 0 & x^5 \\
\end{array}
\right)=x^3a_2.
$$

\section{Concluding Remarks}

In the paper \cite{MilsonKasman2020} the operators having the
generating function for Exceptional Hermite Polynomials as an
eigenfunction were ``translated'' into the recursion relations.   The
main 
goal of this note is to explain and generalize the process by which that took place.
We also presented a procedure by which one can demonstrate
orthogonality of a sequence by using its generating function.  (See
\cite{usedit} and the footnote in the introduction.)
It is hoped that others may find it useful to apply
Theorems~\ref{thm:main} and \ref{thm:orthog} in their own research.

For the sake of clarity, some generalizations and tangential topics
have been ignored in the body of the paper above.  In particular:
\begin{itemize}
\item This paper considered only \textit{ordinary} generating
functions in which the coefficients in the power series expansion are precisely
the sequence of interest.  More generally, it is common to consider
generating functions involving some other known sequence
$\{c_0,c_1,c_2,\ldots\}$ with $c_i\not=0$ so that
$$
\psi(z)=\sum_{n=0}^{\infty} c_na_n z^n.
$$
For instance, many sequences have generating functions of this form
with $c_n=1/(n!)$ or $c_n=1/n$.  The main results could be stated so
as to be valid for these more general generating functions as well.
One could simply rescale the coefficients of $\Omega_L$
in Theorem~\ref{thm:main} in order to handle the general case.
\item 
Similarly, the results could be generalized to handle the case of a
generating function whose Laurent series representation (including
negative powers of the formal variable) has the sequence of interest
as coefficients.
\item In the case of the Exceptional Orthogonal Polynomials, it is
  important that the terms in the sequence which are equal to zero do
  not arise in the recursion relations.  In particular, one wants the
  additional property that all of the terms which appear in the
  difference equations obtained by applying one of the operators
  $\Omega_L$ to a non-zero polynomial in the sequence area also
  non-zero polynomials.  See
  \cite{MilsonKasman2020} for a construction which takes this
  additional restriction into account.
\item The generating functions used in Sections~\ref{sec:examp2} and
  \ref{sec:examp3} also happen to be eigenfunctions for differential
  operators in the variable $x$.  Although this \textit{bispectrality} plays an essential
  role in the papers 
  \cite{MilsonKasman2020} and \cite{Wilson1993}, it does not seem to be a fundamental feature
  of the main result in this paper and hence was not mentioned above.
\end{itemize}
We thank organizers/editors Ahmad Barhoumi, Roozbeh Gharakhloo, and Andrei Martinez-Finkelshein 
for their efforts and thank the other participants in the special session for their talks and
helpful feedback.

\bibliographystyle{amsplain}
\bibliography{kasman-conm.bib}

\end{document}